\documentclass[a4, 12pt]{article}
  \usepackage{amsmath, amssymb, latexsym, amscd, amsthm,amstext}
  \usepackage{authblk}
  \usepackage{hyperref}                            
  \usepackage{enumerate}  
\usepackage{anysize}
\marginsize{1.5in}{1.3in}{0.5in}{0.5in}

  \def\R{\mathbb{ R}}
  \def\C{\mathbb{ C}}
  
  \def\Z{\mathbb{ Z}}

  \def\Sy{\mathbb{ S}}

  \def\bP{\mathbf{ P}} 
  \def\bF{\mathbf{ F}}   
  \def\bD{\mathbf{ D}}  
  \def\bA{\mathbf{ A}}  
  \def\bB{\mathbf{ B}}    
  \def\bC{\mathbf{ C}}    
  \def\bX{\mathbf{ X}}
  \def\bU{\mathbf{ U}}  
  \def\bV{\mathbf{ V}}

  \def\mP{\mathcal{ P}}
  \def\rk{\mbox{\rm rank}}
  \def\tr{\mbox{\rm Tr}}
  
  \def\diag{\mbox{\rm diag}}
  \def\re{\mbox{\rm Re}}
  \def\im{\mbox{\rm Im}}

    \def\jac{\mbox{\rm \texttt{Jac}}}

  \newtheorem{theo}{\bf Theorem}
  
  \newtheorem{coro}[theo]{\bf Corollary}
  
  \newtheorem{prop}[theo]{\bf Proposition}
  \newtheorem{exam}{\it Example}
  \newtheorem{rema}{\it Remark}

  \rmfamily
  
\begin{document}
 \title{\bf 
 Sum-of-square-of-rational-function
based representations of positive semidefinite  polynomial matrices
}

\author[a]{L\^{e} Thanh Hi{\^e}\hskip-1.345mm\'{}\hskip-1mmu\thanks{
\\
\mbox{}
\ \ \quad Email: \texttt{lethanhhieu@qnu.edu.vn}}}
\affil[a]{\small Department of Mathematics, Quy Nhon University, Vietnam}

\author[b]{Ph{\d a}m Nh{\d {\^a}}t Thi{\d {\^e}}n\thanks{
Email: \texttt{nhat-thien.pham@u-psud.fr} }}
\affil[b]{Department of Mathematics, Paris-Sud University, France}

 \maketitle
 
 \begin{abstract}
 The paper 
 proves
 sum-of-square-of-rational-function based representations (shortly, sosrf-based representations)
 of polynomial matrices that are positive semidefinite on some special sets:
 $\R^n;$
$\R$ and its intervals $[a,b],$ $[0,\infty);$
and the strips $[a,b] \times \R \subset \R^2.$
A method for numerically computing such representations is also presented.
The methodology is divided into two stages: 
 (S1) diagonalizing the initial polynomial matrix based on the Schm\"{u}dgen's procedure \cite{Schmudgen09};
 (S2) 
	for  
 each diagonal element of the resulting matrix, 
 find its low rank sosrf-representation
 satisfying the Artin's theorem
solving the Hilbert's 17th problem.  
Some numerical tests and illustrations with \textsf{OCTAVE} are also presented for each type of polynomial matrices.

 \end{abstract}
 




\section{Introduction and preliminary}\label{rmp_intr}


Nonnegative polynomials,
i.e., polynomials with scalar-coefficients,
appears in a variety of mathematical problems and applications.
The problem of approximating such polynomials as sums of squares of polynomials have been studied more deeply
in both theoretical and practical points of view, 
see eg., \cite{Menini2018} and references therein.
In our opinion, this idea can be generalized to matrix polynomials,
i.e., 
polynomials with scalar-matrix coefficients.
A matrix polynomial can also be written as a polynomial matrix,
i.e., a matrix with polynomial entries.
However,
the terminology ``nonnegative'' for scalar polynomials must be generalized as ``positive semidefinite'',
which is usually used for (scalar or polynomial) matrices.
Because of this reason,
we prefer ``polynomial matrices'' to ``matrix polynomial''
 in this paper. 
 
There have been a variety of researches focusing on the generalization of
Positivstellens\"{a}tze 
to
psd-polynomial matrices, 
i.e., $m\geq 2,$
and its applications 
(see in, eg., \cite{Trinh15} and references therein).
However, from the practical point of view,
most applications lead to problems over univariate/bivariate polynomial matrices. 
This motivates us to focus on these two types of polynomial matrices.

Let $\R[x]$ be the ring of $n$-real variable polynomials with real coefficients.
Let 
$\R^{m\times m}$
and $\Sy^m \R$
denote
 the set of all $m\times m$ matrices
with real entries
 and its subset of symmetric matrices, respectively.
By $.^\mathrm{T}$ we denote the transpose of matrices.
For $A \in \Sy^m\R,$
by
$A\succeq 0$ we mean $A$ is positive semidefinite, i.e., $u^\mathrm{T}Au \geq 0$ for all $u\in \R^n.$ 
We denote 
 $\Sy_+^m\R$ 
the set of all positive semidefinite matrices in $\Sy^m\R.$ 
Moreover,
for two matrices $A$ and $B $ in $\Sy^m\R,$ 
we write $A\succeq B$ if $A-B\in \Sy_+^m\R.$
%
For a subset $\mathcal{D} \subseteq \R^n,$
and a polynomial matrix $\bF\in \Sy^m\R[x],$
we say
$\bF$ to be positive semidefinite on $\mathcal{D}$ if
$$
\mathcal{D} = \{x\in \R^n|\ \bF(x) \succeq 0\}
				= \{x\in \R^n|\ u^\mathrm{T}\bF(x)u\geq 0, \forall u\in \R^n\}.
$$
%
We will say 
``\textit{psd}''
instead of ``positive semidefinite''.
Someone calls a psd polynomial matrix on $\mathcal{D} = \R^n$ 
the \textit{global psd} one.
Throughout this paper, we fix 
\begin{itemize}
\item 
	$m,n$ as above, i.e., $m$ is the size of a polynomial matrix and $n$ is the number of variables in the matrix;

\item
	$d:$ the maximum of degrees of $m^2$ polynomials in the matrix.
\end{itemize}
A symmetric polynomial matrix $\bF$ is called
a  ``\textit{sum of squares of polynomial matrices}'' 
(resp., a \textit{sum of squares of rational functions matrices}\footnote{
A \textit{rational function} is a ratio of polynomials.
})
if it is a finitely many sum of the form
$$
\bF = \sum_{i=1}^r \mathbf{A}_i^\mathrm{T} \mathbf{A}_i,
$$
where the matrices $\bA_i$ has polynomial entries
(resp., rational-function entries) 
defined on the whole space $\R^n.$
We shortly say a sum of squares matrix 
a 
\textit{sos-matrix}
while the other  a 
\textit{sosrf-matrix}.
It is clear that 
if $ b^2\bF,$ $b\in \R[x],$ is sos
then
$\bF$ is sosrf since
$$
\bF = \sum_{i=1}^r \left( \frac{1}{b}\mathbf{A}_i^\mathrm{T} \right) 	
						\left( \frac{1}{b}\mathbf{A}_i \right).
$$

It can be seen that any sos-matrix or sosrf-matrix is psd on $\R^n.$
The converse direction is always true for sosrf-matrices since Proposition \ref{prop:coroSchmud},
but not for sos-matrices.
The observation of this issue is easy to see for scalar polynomials 
including
a famous counter example of the Motzkin polynomial
$$
f^M(x,y) = 1 + x^2y^4 + x^4y^2 -3 x^2y^2.
$$
It cannot be expressed as a sum of squares of polynomials in $\R[x,y]$
but
$$
p^2 f^M(x,y) = (x^2-y^2)^2  + [xy(p-2)]^2 + [x^2y(p-2)]^2 + [xy^2(p-2)]^2
$$
with $p(x,y)= x^2+y^2.$

When $m=1,$  polynomial matrix concept here coincides with the scalar polynomial as usual,
and ``psd-polynomial matrices'' now means ``nonnegative polynomials'' as well.
There is  a number of  well-known Positivstellens\"{a}tze for nonnegative/positive scalar polynomials \cite{Krivine64, Stengle74},
leading to representations of 
scalar polynomials nonnegative on several subsets of $\R$
using sos-representations of other polynomials.
For example, 
it is well known in literature that,
see also in eg., \cite{Brickman62},
 any univariate real polynomial nonnegative on $\R$ can be expressed as a sum of two squares of real polynomials.
This is called a \textit{sos-representation} of the initial polynomial.
And,
a polynomial $f$ nonnegative on $[0, \infty)$ 
can be written as
$$
f(x) = p^2(x) + xq^2(x), 
$$
for some real polynomials $p,q.$
We will call
the later representation  the
``\textit{sos-based representation}'' of $f.$

In \cite{LeVaBa15},
the authors propose an algorithm for decomposing a nonnegative polynomial on $\R^n$
as a sum of squares of rational functions
relying on the idea of 
Reznick \cite{Reznick95}, saying that
the common denominator 
is a power of the sum of squares of coordinate functions.
In this paper, 
we deal with the problem of finding a 
sos-based and/or sosrf-based representation of a polynomial matrix which is positive semidefinite over one of the sets:
 $\R^n,$ 
 $[a,b],$
 $[0,+\infty)$
 and strips in $\R^2.$
The idea is to 
combine Schm\"{u}dgen's procedure of diagonalizing polynomial matrices and Levenberg-Marquardt algorithm \cite{r653} 
finding sosrf-representations of
the diagonal polynomials of the resulting diagonal matrices.
Our method for finding sosrf-based representations for psd-polynomial matrices is hence divided into two stages:

(S1) 
	diagonalizing the initial polynomial matrix $\bF$ suggested by Proposition \ref{prop:Schmud} below so that \eqref{eq:SchmudRel} holds true.
	Note that 
	 the initial polynomial matrix is positive semidefinite if and only if
  the resulting diagonal polynomials are all nonnegative;

(S2) 
	finding low-rank sos-representations of the resulting diagonal polynomials by
	applying the algorithm proposed in \cite{LeVaBa15}.
	The representation for the initial polynomial matrix we need is then obtain by substituting to the relation in the first stage.

As mentioned earlier,  
the second stage is 
to find sosrf-representations of scalar polynomials. 
It is well-known that any sum of squares polynomial can be determined by its Gram matrix.
With the help of Artin's theorem answering to the Hilbert's 17th problem, one notes 
\begin{align*}
	f \geq 0 \mbox{ on } \R^n
&\Longleftrightarrow
	\exists b\in \R[x]: b^2f = \sum_{i=1}^r f_i(x)^2, \quad f_i\in \R[x]\\
&\Longleftrightarrow
		\exists b\in \R[x], G \in \Sy^e\R: 
		b(x)^2f(x) = \pi(x)^\mathrm{T} G \pi(x),
\end{align*}
where $e = {n+ \deg(b^2f)\choose n},$
$\pi(x)$ is the vector of monomials $x^\alpha := x_1^{\alpha_1} \ldots x_n^{\alpha_n}$
with degree not exceed $\deg(b^2f),$
and 
$G$ 
is real symmetric and positive semidefinite,
which is called a \textit{Gram matrix} of $b^2f.$
By comparing the coefficients,
we see that each coefficient on the left side is linearly dependent on the entries of $G.$

In addition,
a Gram matrix of a sos-polynomial 
can be found to be low rank. 
This suggests us to
propose and solve
the matrix rank minimization problem 
\begin{equation}\label{eq:Frmp}
\mbox{minimize} \left\{ \rk(X) \ | \  X\in \R^{m\times p}, \phi(X) = u \right\},
%
\end{equation} 
\noindent
where 
$\phi: \R^{m\times p} \rightarrow \R^l$ is a differentiable  map,
that can be applied in our second stage.
The rank function is non-convex,
even the feasibility region is smooth.
The idea of this stage is to check whether the problem (\ref{eq:Frmp}) has a numerical solution of rank $r,$
step by step, for $r=1,2,\ldots$
At each step, with respect to each $r,$
the Levenberg-Marquardt method (for short, \textit{LM-method}) 
\cite{r653}
solving least square problems 
is applied to the function 
defined by $\phi(X)= u.$
The differentiability of $\phi$ guarantees the existence of the Jacobian matrix of $\phi$ in Levenberg-Marquardt iterations. 

The problem of finding a sos-representation of a scalar sos-polynomial is thus equivalent to which of solving a linear system finding a polynomial's Gram matrix with low rank.
This is reasonable to apply the model (\ref{eq:Frmp}) 
with respect to a suitable 
linear map $\phi.$
Furthermore,
any polynomial's Gram matrix must be symmetric,
and the existence of a Gram matrix is equivalent to which of its Cholesky factor\footnote{
A positive semidefinite matrix $X$ can always be factorized as
$X = YY^\mathrm{T},$ where $Y$ is called a Cholesky factor of $X.$
},
where both are the same rank.
In this sense, 
$\phi$ will be a quadratic function in the Cholesky factor.
Our second stage, as seen in Section \ref{sec:RankMod}, 
deals properly with finding such Cholesky factors.
The polynomial's Gram matrix is then derived, and so is the polynomial.
This is done by applying the LM-method.

The LM-method has been generalized for complex setting in \cite{r665, cot}, which is called the generalized LM-method, or gLM-method.
It solves the
least squares problems with real-valued function in complex variables.
The corresponding \textsc{Matlab} toolbox is called COT  \cite{cot}.
The authors there
convert the complex setting to the real one by
using 1-1 linear transformations that map any complex $n$-tuples 
$z  \in \C^n,$
to 
$(\re(z), \im(z)) \in \R^{2n},$
and/or 
to
$(z,\bar{z})$
with the help of Wirtinger calculus and the complex Taylor series expansions.
Even 
the gLM-method aims at solving least square problems with
real-valued functions in complex variables, 
it is also valid for real-valued functions in real variables.
In our work we recode the gLM-method 
in \textsf{OCTAVE} that makes more convenience for the readers.
Unlike as
 ``\texttt{lsqnonlin}'' function in \textsc{Matlab},
\textsf{OCTAVE} codes,
as those of COT as well, 
additionally need the Jacobian/gradient of the function as an input.
We hence give detailed formulas for Jacobian matrices with respect to particular cases 
that gLM-method  needs.
The numerical tests in this paper are all implemented in \textsf{GNU OCTAVE} 4.4.2\footnote{Available at \texttt{https://octave.sourceforge.io/}} 
and Ubuntu 14.04,
on a desktop 
Intel(R) Core(TM) i3-3220 CPU@3.30GHz and 4GB RAM.

This paper is organized as follows.
Section \ref{sec:SchmudDiag} 
gives a clear explaination to 
Schm\"{u}dgen diagonalization, 
which is considered as the first stage of our method,
and  
presents the corresponding algorithm. 
Some numerical tests are also illustrated.
The second stage is based on a rank matrix minimization model 
devoted in Section  \ref{sec:RankMod}.
After proposing an algorithm solving such a model, 
we apply to solve our second stage.
We give some numerical examples as well.
In section \ref{sec:sosrf-based}
we prove sosrf-based representations for polynomial matrices 
that are psd on the sets
$\R^n,$ ($n\geq 1$),
some intervals in $\R,$
and some 
strips in $\R^2.$
The conclusion of the paper is devoted in the last section.


\section{Diagonalization of polynomial matrices} \label{sec:SchmudDiag}

\subsection{Schm\"{u}dgen's procedure}
We start this section by recalling a very important matrix decomposition 
by Schm\"{u}dgen \cite{Schmudgen09} 
which is a key part for numerically computing 
sum of Hermitian squares of polynomial matrices below.
For a symmetric matrix $F $  
with entries in a commutative unitary ring $R,$
set
\begin{equation}\label{eq:SchmudProc}
X_\pm = 
\begin{bmatrix}
\alpha &0 \\ \pm \beta^\mathrm{T} & \alpha I
\end{bmatrix}, \quad
\widetilde{F} =
\begin{bmatrix}
\alpha^3 &0 \\ 0  & \alpha (\alpha C - \beta^\mathrm{T}\beta )
\end{bmatrix}
\in R^{m\times m},
\end{equation}
where $\alpha,$ $\beta$ and $C$  are parts of $F$ in the following partition
$$
F = 
\begin{bmatrix}
\alpha & \beta \\ \beta^\mathrm{T} & C
\end{bmatrix}.
$$
Then the following relations  hold
\begin{equation}\label{eq:SchmudRel}
X_+ X_- = X_- X_+ = \alpha^2 I, \quad
\alpha^4 F = X_+ \widetilde{F} X_+^\mathrm{T}, \quad
\widetilde{F}= X_- F X_-^\mathrm{T}.
\end{equation}

Schm\"{u}dgen's procedure \eqref{eq:SchmudProc} is applied to 
diagonalize a polynomial matrix as in the following proposition, 
viewed as an extension of Artin's theorem solving Hilbert's 17th problem.

\begin{prop}{\rm \cite[Corollary 10]{Schmudgen09}}\label{prop:Schmud}
Let $\bF\in \Sy^m \R[x] .$
Then there exist polynomial matrices 
$\bX_+, \bX_- \in \R[x]^{m\times m}$ 
and polynomials $b,d_j \in \R[x],$ $j=1, \ldots, m,$ 
such that
\begin{equation}\label{eq:SchmudDec}
\bX_+ \bX_- = \bX_- \bX_+ = bI_m, \quad
b^2 \bF = \bX_+ \diag(d_1, \ldots, d_m) \bX_+^\mathrm{T}
\end{equation}
and
$X_- \bF X_-^\mathrm{T} = \diag(d_1, \ldots, d_m).$
\end{prop}

\subsection{Algorithmetic implementation}\label{sec_NumImp}

Even though Schm\"{u}dgen's proof (of Proposition \ref{prop:Schmud}) suggests us an algorithm to find 
$b, \bX_{\pm}$ and $\bD,$ 
it was not numerically clarified in his seminal paper. 
In this section, 
we  
explain in more detailed 
how 
Schm\"{u}dgen's procedure \eqref{eq:SchmudProc}
is applied
to obtain 
$b, \bX_{\pm}$ and $\bD$ as in Proposition \ref{prop:Schmud},
step by step such that 
one can implement the algorithm in any scientific programming language.

Given a polynomial matrix $\bF\in \Sy^m \mathbb{R}[t]$ 
that is partitioned as  
$$
\bF = 
\left[
	\begin{array}{cc}
		\alpha	&	\beta\\
		\beta^\mathrm{T}	&	\bC
	\end{array}
	\right],
\text{ where } 	\bC 		\in \Sy^{m-1} \mathbb{R}[t],\quad
				\beta 	\in \mathbb{R}[t]^{1\times (m-1)}.
$$
%
With the help of the procedure in \eqref{eq:SchmudProc},
our algorithm below starts with 
$ \bF= [f_{ij}^{(0)}]  \in \Sy^m\R[t]$
and 
\begin{equation}\label{eq:Iter0}
\alpha_0 = f_{11}^{(0)},\quad
\bX_{0\pm} = 
\left[
		  \begin{array}{cc}
				\alpha_0 	& 0\\
				\pm\beta_0^\mathrm{T}	& \alpha_0 I
		  \end{array}
 \right],\quad
\widetilde{\bF}_{0}  = 
				\left[
		  		\begin{array}{cc}
					\alpha_0^3 	& 0\\
					0			& \bB_0
		  		\end{array}
		  		\right],
		  		\enskip
		  		\bD^{(0)} = \widetilde{\bF}_{0},
\end{equation}
where
$\bB_0=\alpha_0^2 \bC_0-\alpha_0\beta_0^\mathrm{T}\beta_0\in \Sy^{m-1}\R[t].$
The next iteration will apply 
procedure \eqref{eq:SchmudProc}
to the matrix $\bB_0.$

We would emphasize that, afterwards,
the matrices with subscripts
$\bX_{i+}, \bX_{i-},$
$\bF_i \in \Sy^{m-i}\R[t]$ have sizes 
that are changed through the different iterations;
while
the others,
indexed by superscripts such as
$b^{(i)}, \bX_{\pm}^{(i)},$ $\bD^{(i)},$
will be updated
from $\bX_{i+}, \bX_{i-}, \bF_i$
as in Proposition \ref{prop_alg}.
The matrices with superscripts  
must have the same size as $\bF$ at every iteration.
The algorithm proceeds by induction (on $m$) until 
$\bB_i\in \Sy^1 \mathbb{R}[t]$ or $\bB_i$ is identically zero,
i.e., 
$\bD^{(i)}$ is diagonal
and the outputs are 
$\bD= \bD^{(i)},$
$ \bX_{\pm} = \bX_{\pm}^{(i)}, b = b^{(i)}.$

From the computational point of view, 
the matrix $\bD$ should be treated as 
its diagonal vector.
We thus only need to compute its main diagonal element by element. 
In other words, 
at Iteration $i,$ for instant, 
we only need to compute $d_i$ (exactly equals to $\alpha_i^3$) and update $d_0, d_1,\ldots,d_{i-1},$ 
but do not pay attention for $d_r, i< r\leq m.$ 


%
%
%

Below, 
we will see in our algorithm that 
the data at iteration $i$ ($i\geq 2$) is updated 
from that at iteration $i-1.$
More precisely,
suppose we are given  
\begin{equation}\label{eq:Iteri}
\alpha_{i-1}, \enskip
\bX_{\pm}^{(i-1)} 
 ,\enskip
\bD^{(i-1)}=
				\left[
		  		\begin{array}{cc}
					\diag(d_0^{(i-1)}, \ldots, d_{i-1}^{(i-1)}) 	& 0\\
					0			& \bB_{i-1}
		  		\end{array}
		  		\right],
\end{equation}
where
$
\bB_{i-1}=\alpha_{i-1}^2 \bC_{i-1}-\alpha_{i-1}\beta_{i-1}^\mathrm{T}\beta_{i-1}\in \Sy^{m-i+1}\R[t].$
Then by applying procedure \eqref{eq:SchmudProc} to
$\bF_{i} = \widetilde{\bB}_{i-1} =[f_{rs}^{(i)}]$
we compute
\begin{equation}\label{eq:PreIteri1}
\alpha_{i} = f_{11}^{(i)}, \enskip
\bX_{i\pm} = 
\left[
		  \begin{array}{cc}
				\alpha_{i} 	& 0\\
				\pm\beta_i^\mathrm{T}	& \alpha_i I
		  \end{array}
 \right],\enskip
\widetilde{\bF}_i=
				\left[
		  		\begin{array}{cc}
					\alpha_i^3 	& 0\\
					0			& \bB_i
		  		\end{array}
		  		\right] \in \Sy^{m-i}\R[t] 
\end{equation}
that satisfy the relation \eqref{eq:SchmudRel}:
\begin{equation}\label{eq:PreIteri2}
 \bX_{i+}\bX_{i-} = \bX_{i-}\bX_{i+} = \alpha_i^2 I,\quad
 \alpha_i^4 \bF_i = \bX_{i+} \widetilde{\bF}_i \bX_{i+}^\mathrm{T},\quad
 \widetilde{\bF}_i = \bX_{i-} \bF_i \bX_{i-}^\mathrm{T}.
\end{equation}	
One then updates
\begin{equation}\label{eq:Iteri1}
b^{(i)},\quad
\bX_{\pm}^{(i)}, \quad
\bD^{(i)}=
				\left[
		  		\begin{array}{cc}
					\diag(d_0^{(i)}, \ldots, d_{i}^{(i)}) 	& 0\\
					0			& \bB_{i}
		  		\end{array}
		  		\right]
		  		\in \Sy^m\R[t],
\end{equation}
where $d_k^{(i)}$'s are determined as in Proposition \ref{prop_alg} below.

To unify the notations and to make the readers more convenient,
we 
implement two first iterations as follows.


\noindent
\textbf{Iteration 0.} 
\textit{Define 
$\alpha_0, \bX_{0\pm},$ $\bB_0\in \Sy^{m-1}\R[t]$ and $\widetilde{\bF}_0 \in \Sy^{m}\R[t]$}
from $\bF$ as in \eqref{eq:Iter0}, then put
$$
b^{(0)} = \alpha_0^2,
\quad \bX_{\pm}^{(0)} = \bX_{0\pm} 
\text{ and } 
\bD^{(0)} =\widetilde{\bF}_0 \in \Sy^m\R[t],\quad
d_0^{(0)} = \alpha_0^3.
$$
It is clear that
\eqref{eq:SchmudRel} holds:
\begin{equation}\label{eq:RelIter0}
\bX_+^{(0)} \bX_+^{(0)} = \bX_-^{(0)} \bX_+^{(0)} = \alpha_0^2 I_m,
\enskip
\alpha_0^4 \bF = \bX_+^{(0)} \bD^{(0)} [\bX_+^{(0)}]^\mathrm{T}, 
\enskip
\bD^{(0)}  = \bX_-^{(0)} \bF_0 [\bX_-^{(0)}]^\mathrm{T}.
\end{equation}

\noindent
\textbf{Iteration 1.} 
Set $\bF_1= \bB_0$ and
determine
$\alpha_1, \bX_{1\pm},$ 
$\widetilde{\bF}_1 =
\begin{bmatrix}
\alpha_1^3 & 0 \\ 0 & \bB_1
\end{bmatrix} \in \Sy^{m-1}\R[t]$
by applying procedure \eqref{eq:SchmudProc} to $\bF_1.$
These satisfy relation \eqref{eq:SchmudRel}:
$$
\bX_{1+} \bX_{1-} = \bX_{1-}\bX_{1+} = \alpha_1^2, \enskip
\alpha_1^4 \bF_1 = \bX_{1+} \widetilde{\bF}_1 \bX_{1+}^\mathrm{T}, \enskip
\widetilde{\bF}_1 = \bX_{1-} \bF_1  \bX_{1-}^\mathrm{T}.
$$
This, combining with \eqref{eq:RelIter0}, implies that
\begin{align}
\alpha_1^4 (\alpha_0^4 \bF) &= \bX_{0+}(\alpha_1^4 \bD^{(0)}) \bX_{0+}^\mathrm{T}		\nonumber\\
&= \bX_{0+}
	\left[
	\begin{array}{cc}
	\alpha_1^4\alpha_0^3	&	0\\
	0	&	\bX_{1+}\widetilde{\bF}_1 \bX_{1+}^\mathrm{T}
	\end{array}
	\right] 
	\bX_{0+}^\mathrm{T}		\nonumber\\
&=\bX_{0+}
	\left[
	\begin{array}{cc}
	\alpha_1	&	0\\
	0			&	\bX_{1+}
	\end{array}
	\right]
	\left[
	\begin{array}{cc}
	\alpha_1^2\alpha_0^3	&	0\\
	0						&	\widetilde{\bF}_1
	\end{array}
	\right]
	\left[
	\begin{array}{cc}
	\alpha_1	&	0\\
	0			&	\bX_{1+}^\mathrm{T}
	\end{array}
	\right]
	\bX_{0+}^\mathrm{T}\label{step1X+}
\end{align}
and
\begin{align}
\left[
\begin{array}{cc}
	\alpha_1^2\alpha_0^3	&	0\\
	0	&	\widetilde{\bF}_1
	\end{array}
	\right]
&=
\left[
\begin{array}{cc}
	\alpha_1^2\alpha_0^3	&	0\\
	0	&	\bX_{1-} \bF_1  \bX_{1-}^\mathrm{T}
\end{array}
\right] \nonumber \\
&=
\left[
\begin{array}{cc}
	\alpha_1	&	0\\
	0			&	\bX_{1-}
\end{array}
\right]
\left[
\begin{array}{cc}
	\alpha_0^3	&	0\\
	0			&	\bF_1
\end{array}
\right]
\left[
\begin{array}{cc}
	\alpha_1	&	0\\
	0			&	\bX_{1-}^\mathrm{T}
\end{array}
\right]	\nonumber \\
&=
\left[
\begin{array}{cc}
	\alpha_1	&	0\\
	0			&	\bX_{1-}
\end{array}
\right]
\bX_{0-} \bF_0 \bX_{0-}^\mathrm{T}
\left[
\begin{array}{cc}
	\alpha_1	&	0\\
	0			&	\bX_{1-}^\mathrm{T}
\end{array}
\right]	.	\label{step1X-}
\end{align}
We hence obtain that
 $b^{(1)} = \alpha_1^2\alpha_0^2$ and
\begin{align*}
\bX_+^{(1)} &= \enskip
					{\bX_0}_+ 
					\left[
						\begin{array}{cc}
						\alpha_1	& 0\\
						0			& {\bX_1}_+
					 	\end{array}
					 	\right] \enskip
				= \enskip
					\left[
						\begin{array}{cc}
						\alpha_1\alpha_0	& 0\\
						-{\bX_1}_-\beta_0^\mathrm{T}	& \alpha_0 {\bX_1}_-
			         \end{array}
			      \right],\\
\bX_-^{(1)} &= \enskip
				\left[
					\begin{array}{cc}
					\alpha_1	& 0\\
					0			& {\bX_1}_-
				  \end{array}
				  \right]
				  {\bX_0}_- \enskip
				=	\enskip
						\left[
						\begin{array}{cc}
						\alpha_1\alpha_0	& 0\\
						-{\bX_1}_-\beta_0^\mathrm{T}	& \alpha_0 {\bX_1}_-
			  			\end{array}
			  		\right]				  ,\\
\bD^{(1)} &= \enskip
				\left[
			\begin{array}{cc}
			\alpha_0^3\alpha_1^2	& 0\\
			0						& \widetilde{\bF}_1
			\end{array}
			\right]
			=
\left[
\begin{array}{ccc}
	\alpha_1^2\alpha_0^3	&	0\\
	0							& \alpha_1^3 &0\\
	0	&				0		&\bB_1
	\end{array}
	\right]			
			, \enskip
\begin{bmatrix}
d_0^{(1)}\\
d_1^{(1)}
\end{bmatrix}
=\begin{bmatrix}
\alpha_1^2 \alpha_0^3\\
\alpha_1^3
\end{bmatrix}
= \alpha_1^2
\begin{bmatrix}
d_0^{(0)}\\
\alpha_1
\end{bmatrix}
.			
\end{align*}
Provided by
\eqref{step1X+} and \eqref{step1X-} ,
 $\bX_{\pm}^{(1)}$ and $\bD^{(1)}$ 
 also satisfy the relation \eqref{eq:SchmudRel}:
$$
\bX_+^{(1)} \bX_-^{(1)} = \bX_-^{(1)} \bX_+^{(1)} = b^{(1)}I_m, \enskip
[b^{(1)}]^2\bF=\bX_+^{(1)} \bD^{(1)} [\bX^{(1)}]^\mathrm{T}_+, \enskip
\bD^{(1)}=\bX_-^{(1)} \bF  [\bX_-^{(1)}]^\mathrm{T}.
$$
%
%
The diagonal of $\bD^{(1)},$ as mentioned earlier,
should be 
updated only $d_1^{(1)}=d_1^{(0)}\alpha_1^2$
 such that $\bD^{(1)}$ has the form \eqref{eq:Iteri}.

The two early iterations show us the numerical formulas to compute $b, \bX_{\pm}$ and $\bD$ in the next iterations such that 
they satisfy \eqref{eq:SchmudRel}. 
The mathematical explanation 
in more detail, compared with \eqref{eq:Iteri}, \eqref{eq:PreIteri1} and \eqref{eq:Iteri1},
is showed in the following proposition.

\begin{prop}\label{prop_alg}
With notation as in \eqref{eq:Iteri}, \eqref{eq:PreIteri1} and \eqref{eq:Iteri1}.
Suppose the data at Iteration $i-1$ with $i\geq 1$
is already given: 
$
b^{(i-1)}, \bX_{\pm}^{(i-1)}
$
and
$\bD^{(i-1)}.$
Then for all 
$i=2,\ldots, m-1,$
we have the update rules
\begin{align} \label{eq:update}
b^{(i)} 			&= b^{(i-1)}\alpha_i^2,  \nonumber \\ 
\bX^{(i)}_+ 	& = 				
							\bX_+^{(i-1)}\cdot
							\left[
								\begin{array}{cccc}
								\alpha_i	I_i&		&			&0\\
								0			&		&			&\bX_{i+}
								 \end{array}
							\right],\quad
\bX^{(i)}_-  = 
							\left[
								\begin{array}{cccc}
								\alpha_i	I_i&		&			&0\\
								0			&		&			&\bX_{i-}
								 \end{array}
							\right]
							\cdot 
							\bX_-^{(i-1)},			 \\
\bD^{(i)} &= 
						\left[
							\begin{array}{clll}
							\alpha_i^2 \cdot \diag\left(d_0^{(i-1)}, \ldots, d_{i-1}^{(i-1)}, \alpha_i \right)	&		&			&0\\
							0			&		&			& 					
							\bB_i
			 				\end{array}
			 			\right]. \nonumber
\end{align}
These matrices furthermore 
satisfy the relation described in \eqref{eq:SchmudRel}:
\begin{align}\label{eq:Prop} 
\bX^{(i)}_+\bX^{(i)}_- = \bX^{(i)}_-\bX^{(i)}_+ = b^{(i)}I,\enskip
{b^{(i)}}^2 \bF=\bX^{(i)}_+\bD^{(i)} {\bX_+^{(i)}}^{ \mathrm{T}},\enskip
\bD^{(i)}=\bX^{(i)}_- \bF {\bX^{(i) }_-}^{\mathrm{T}}.
\end{align}
\end{prop}

\begin{proof}
Since the section of Iteration 1, 
the proposition is true with $i=1.$ 
Now suppose that the proposition is true for some 
$i\geq 1,$ 
namely,
$$
\bX_+^{(i)}\bX_-^{(i)} = \bX_-^{(i)}\bX_+^{(i)} = b^{(i)}I_m,\enskip 
{b^{(i)}}^2 \bF =\bX_+^{(i)}\bD^{(i)}{\bX^{(i)}_+}^\mathrm{T},\enskip
\bD^{(i)}=\bX^{(i)}_- \bF{\bX^{(i)}_-}^\mathrm{T},
$$
where $b^{(i)} = b^{(i-1)}\alpha_{i}^2 = \alpha_1^2\alpha_2^2\ldots\alpha_{i}^2,$
\begin{align*}
\bX_+^{(i)} = \bX_+^{(i-1)}\cdot
							\left[
								\begin{array}{cccc}
								\alpha_i	I_i&		&			&0\\
								0			&		&			&\bX_{i+}
								 \end{array}
							\right], \quad
\bX_-^{(i)} = 
							\left[
								\begin{array}{cccc}
								\alpha_i	I_i&		&			&0\\
								0			&		&			&\bX_{i-}
								 \end{array}
							\right]
			 		\cdot \bX_-^{(i-1)}
\end{align*}
and
$\bD^{(i)}$ is defined as in \eqref{eq:update}. 
That is
$$
d_k^{(i)} = \alpha_i^2 d_k^{(i-1)} \mbox{ for } 0\leq k\leq i-1,
\quad\mbox{and} \quad
d_i^{(i)} =\alpha_i^3.
$$
We will show that \eqref{eq:Prop} also holds for $i+1.$ 
By inductive hypothesis, 
with noting that  
${b^{(i+1)}} = \alpha_{i+1}^2 {b^{(i)}},$
one has
\begin{align*}
{b^{(i+1)}}^2 \bF 
&= \bX_{+}^{(i)} \left( \alpha_{i+1}^4 \bD^{(i)}\right) \bX_+^{(i)}
\\
&= \bX_{+}^{(i)}
	\left[
	\begin{array}{rl}
	\alpha_{i+1}^4 \alpha_i^2 \diag(d_0^{(i)}, \ldots, d_{i-1}^{(i)}, \alpha_i) \enskip	&	0\\
	0	\enskip &	\alpha_{i+1}^4 \bB_i
	\end{array}
	\right] 
	{\bX_{+}^{(i)}	}^\mathrm{T},
	\quad
	(\bB_i = \bF_{i+1})	\\
&= \bX_{+}^{(i)}
	\left[
	\begin{array}{rl}
	\alpha_{i+1}^4 \alpha_i^2 \diag(d_0^{(i)}, \ldots, d_{i-1}^{(i)}, \alpha_i) \enskip	&	0\\
	0	\enskip &	\bX_{(i+1)+} \widetilde{\bF}_{i+1} \bX_{(i+1)+}^\mathrm{T}
	\end{array}
	\right] 
	{\bX_{+}^{(i)}}^\mathrm{T}	,\\
&=\bX_{+}^{(i)}
	\left[
	\begin{array}{rl}
	\alpha_{i+1} I_{i+1} \enskip	&	0\\
	0	\enskip		&	\bX_{(i+1)+}
	\end{array}
	\right]
	\left[
	\begin{array}{rl}
	\alpha_{i+1}^2 \alpha_i^2 \diag(d_0^{(i)}, \ldots, d_{i-1}^{(i)}, \alpha_i) \enskip 	&	0\\
	0	\enskip					&	\widetilde{\bF}_{i+1}
	\end{array}
	\right] \times \\
	&\ \quad 
	\left[
	\begin{array}{rl}
	\alpha_{i+1} I_{i+1} \enskip	&	0\\
	0	\enskip		&	\bX_{(i+1)+}^\mathrm{T}
	\end{array}
	\right]	
	{\bX_{+}^{(i)}}^\mathrm{T}
	\\
&= \bX_{+}^{(i+1)}
	\left[
	\begin{array}{ccc}
	\alpha_{i+1}^2 \alpha_i^2 \diag(d_0^{(i)}, \ldots, d_{i-1}^{(i)}, \alpha_i) \enskip	&	0   & 0\\
	0  	& \alpha_{i+1}^3  & 0\\
	0		& 0	\enskip &	\bB_{i+1}
	\end{array}
	\right] 
	{\bX_{+}^{(i+1)}}^\mathrm{T}
	\\
&= \bX_{+}^{(i+1)} \bD^{(i+1)} {\bX_{+}^{(i+1)}}^\mathrm{T}.	
\end{align*}
This yields the form of $\bD^{(i+1)}$ and $\bX_{+}^{(i+1)}.$

For the last relation,
$\bX^{(i+1)}_-\bF{\bX^{(i+1)}_-}^\mathrm{T}= \bD^{(i+1)},$
we write
\begin{align*}
&\quad	
\bD^{(i+1)} =
	\left[
	\begin{array}{rl}
		\alpha_{i+1}^2 \alpha_i^2 \diag(d_0^{(i)}, \ldots, 			
		d_{i-1}^{(i)}, \alpha_i) \enskip 	&	0\\
		0	\enskip					&	\widetilde{\bF}_{i+1}
	\end{array}
	\right]\\ 
&=
	\left[
	\begin{array}{rl}
	\alpha_{i+1}^2 \alpha_i^2 \diag(d_0^{(i)}, \ldots, d_{i-1}^{(i)}, \alpha_i) \enskip 	&	0\\
	0	\enskip					&	\bX_{(i+1)-}\bF_{i+1} \bX_{(i+1)-}^\mathrm{T}
	\end{array}
	\right]	 \\
& =
	\left[
	\begin{array}{rl}
	\alpha_{i+1} I_{i+1} \enskip	&	0\\
	0	\enskip		&	\bX_{(i+1)-}
	\end{array}
	\right]
\underbrace{
	\left[
	\begin{array}{rl}
	 \alpha_i^2 \diag(d_0^{(i)}, \ldots, d_{i-1}^{(i)}, \alpha_i) \enskip 	&	0\\
	0	\enskip					&	\bF_{i+1} 
	\end{array}
	\right]
	}_{=\  \bD^{(i)}, \mbox{ \small as in \eqref{eq:update}}.}
	\left[
	\begin{array}{rl}
	\alpha_{i+1} I_{i+1} \enskip	&	0\\
	0	\enskip		&	\bX_{(i+1)-}^\mathrm{T}
	\end{array}
	\right].
\end{align*}
But 
$\bD^{(i)} = \bX_-^{(i)} \bF \bX_-^{(i)}$ provided by \eqref{eq:Prop}, 
we can conclude
$$
\bD^{(i+1)} =
	\left[
	\begin{array}{rl}
	\alpha_{i+1} I_{i+1} \enskip	&	0\\
	0	\enskip		&	\bX_{(i+1)-}
	\end{array}
	\right]
\bX_-^{(i)} \bF \bX_-^{(i)}
	\left[
	\begin{array}{rl}
	\alpha_{i+1} I_{i+1} \enskip	&	0\\
	0	\enskip		&	\bX_{(i+1)-}^\mathrm{T}
	\end{array}
	\right].
$$
Thus
$$
\bX_-^{(i+1)} =
	\left[
	\begin{array}{rl}
	\alpha_{i+1} I_{i+1} \enskip	&	0\\
	0	\enskip		&	\bX_{(i+1)-}
	\end{array}
	\right]
	\bX_-^{(i)},
$$
and we are done.
\end{proof}

\begin{rema}\rm
We would note that 
$\bD^{(i)}$ in Proposition \ref{prop_alg} can also expressed via 
$\bD^{(i-1)}$
as
$$
\bD^{(i)} = \bD^{(i-1)}\cdot\left(\begin{array}{llll}
							\alpha_i^2	&		&			&0\\
										&\ddots	&			&\\
										&		&\alpha_i^2 	&\\
							0			&		&			& \bF_i^{-1}\widetilde{\bF}_i
			 \end{array}\right),
$$
if
$\bF_i$
is invertible. 
\end{rema}

Proposition \ref{prop_alg} gives us formulas updating the outputs $b, \bX_{\pm}$ and $\bD$ in each step of the algorithm. 
The algorithm supposes that $\alpha_i\neq 0$ at all iterations $i.$
Otherwise, 
according to \cite{Schmudgen09}, 
one can find an orthogonal matrix 
$\mathbf{T}$ 
such that
the $(1,1)$-th entry of 
$\mathbf{T}\bF_i \mathbf{T}^\mathrm{T}$
is nonzero.
A more clear formula can also be found in \cite[Proof of Proposition 5]{Cimpric2012}.

\noindent
\rule[0ex]{\linewidth}{1.5pt}
\noindent \textbf{Algorithm 1}. Dialgonalizing a polynomial matrix.  \\ 
\rule[1ex]{\linewidth}{1.5pt}

\noindent 
\textsc{Input}: \hskip3mm
	 $0\neq \bA\in \Sy^m \mathbb{R}[t] , m \geq2$
	 
\noindent 
\textsc{Output}:
	polynomials $b, d_j \in \mathbb{R}[t], j= \overline{1,m},$ 
	and matrices $\bX_{+}, \bX_{-}\in \mathbb{R}[t]^{m\times m}$
\vskip3mm
\noindent
$\bullet$ \textsf{Assign:}
$\alpha_0 ,$
$\bX_{\pm} = \bX_{0,\pm},$
$\widetilde{\bF}_0 = \widetilde{\bD}^{(0)}$
and
$\bB_0$
as in \eqref{eq:Iter0}.
	 
\noindent
$\bullet$ \textsf{At Iteration} $i\geq 1:$	
\begin{enumerate}[\quad \it Step 1:]
\item 
	Set $\bF_i = \bB_{i-1}.$
\item Determine $\alpha_i, \bX_{i-}, \bX_{i+}, \widetilde{\bF}_i$ from $\bF_i$ as in \eqref{eq:PreIteri1} which satisfy \eqref{eq:PreIteri2}.

\item Update $b^{(i)}, \bX^{(i)}_{\pm}, \bD^{(i)}$ as in \eqref{eq:update}.


\item If $\bB_{i}\in \Sy^1 \mathbb{R}[t]$ or $\bB_i$ is zero 
		then stop.\\
	  Else set $i=i+1$ then go to Step 1.
\end{enumerate}
\vskip-5mm
\rule[0ex]{\linewidth}{1.5pt}

\vskip3mm
In the case Algorithm 1 stops at 
Iteration $i$ with
$0\neq \bB_i \in \Sy^1\R[t],$ 
it defines the $(i+1)$-th diagonal element in $\bD^{(i+1)}$
as $\bB_i.$ 
In this case we set this element as $\alpha_{i+1}.$
It is clear that $\bF$ is positive semidefinite on $\R^n$
if and only if the diagonal elements $d_j$'s are nonnegative on $\R^n.$
From Algorithm 1, 
$d_j$'s can be replaced by more simple polynomials as follows.

\begin{coro}\label{coro:kgSets}
The polynomial matrix $\bF$ in Proposition \ref{prop:Schmud} is positive semidefinite on $\R^n$ if and only if
$\alpha_0, \alpha_1, \ldots, \alpha_{i+1},$
$i=0,1,\ldots, m-1,$
are nonnegative on $\R^n.$
In other words,
\begin{align*}
\bF \succeq 0 \mbox{ on } \R^n\
&\Longleftrightarrow \enskip
	d_j \geq 0 \mbox{ on } \R^n, \forall j=1,\ldots,m\\
&\Longleftrightarrow \enskip
	\alpha_j \geq 0 \mbox{ on } \R^n, \forall j=0,\ldots,i+1.	
\end{align*}
\end{coro}

Another corollary of Algorithm 1 is as follows.
Cimpri\v{c} \cite[Proposition 5]{Cimpric2012}, see also in \cite{Trinh15}, 
proves that for a set $\mathfrak{F}$ of polynomial matrices in  $\Sy^m\R[x],$ one can find a set of scalar polynomials
$\mathfrak{F}'\subset \R[x]$ such that
$$
\left\{
x\in \R^n|\ \bF(x) \succeq 0, \forall \bF\in \mathfrak{F}
\right\}
=
\left\{
x\in \R^n|\ f(x) \geq 0, \forall f\in \mathfrak{F}'
\right\}.
$$
Even though his proof is suggested by Schm\"{u}dgen's procedure,
it is only inductive on matrices of smaller size than $m,$
so that the resulting set of scalar polynomials is not much clear.
With the help of Corollary \ref{coro:kgSets}, such a set of scalar polynomials can be verified as
\begin{equation}\label{eq:kgSets}
\left\{
	x\in \R^n|\ \bF(x) \succeq 0, \forall \bF\in \mathcal{F}
\right\}
=
\bigcap_{\bF \in \mathcal{F}}
 \left\{
        x\in \R^n \Big|
        {\scriptsize
        \begin{array}{ll}
        \alpha_j^{(\bF)}(x) \geq 0, & \alpha_j^{(\bF)}\in \R[x],\\
								&  j=0,\ldots, i_{\mathcal{F}}+1
        \end{array} 
        }	
\right\}.
\end{equation}

\subsection{Numerical illustrations}

In \textsf{OCTAVE}, 
a polynomial will be stored as a vector of its coefficients.
For example, the polynomial $f(t) = t^2$ will be saved as the vector
$f= [0 \enskip 0 \enskip 1].$ 
The syntax ``\texttt{conv}'' of vectors 
was used to compute the multiplication of polynomials.
A polynomial matrix will be saved as a cell,
and the multiplication of matrices will be done as the cell multiplication.

\subsubsection{Univariate polynomials}

Table \ref{tab:SchmudTest} shows some numerical tests for Schm\"{u}dgen's decomposition 
of univariate polynomial matrices.
\begin{table}[h]
\centerline{
\begin{tabular}{l 	l  |   c| c| c| l }
\\
 $m$    	&  	$d	$
 &  Err. of   
 &  Err. of
 &  Err. of
 & CPU time 
\\
 &  	
 &  $(\bX_+\bX_- -bI)$          
 &  $(\bX_-\bF\bX_-^{\mathrm{T}} - \bD)$         
  &  $(\bX_+\bD\bX_+^{\mathrm{T}} -b^2\bF)$ 
 & (seconds)
\\ \hline \hline
2 	& 10   	& 1.6e-16	   &  2.9e-16  & 1.3e-16  & 0.020082\\
2 	& 50   	& 1.8e-16	& 2.6e-16   & 1.9e-16  & 0.023105\\
2 	& 100   	& 3.0e-16	&  3.7e-16  & 2.8e-16  & 0.026690\\
2 	& 1000  	& 1.0e-15	&  1.4e-15  & 8.5e-16  & 0.087682\\
3 	& 10  	& 3.8e-16	&  4.9e-16  & 1.9e-16  & 0.051460\\
3 	& 50  	& 1.4e-15	&  1.4e-15  & 3.9e-16  & 0.063355\\
3 	& 100 	& 1.2e-15	&  1.4e-15  & 4.5e-16 & 0.068851\\
3 	& 1000 	& 3.6e-15	&  4.5e-15  & 1.5e-15  & 0.96089\\
5 	& 50 	   & 6.7e-15	&  6.3e-15  & 1.1e-15  & 0.65835\\
\end{tabular}
}
\caption{Numerical tests for Schm\"{u}dgen's decomposition \eqref{eq:SchmudDec}
with \textsf{OCTAVE}.
} \label{tab:SchmudTest} 
\end{table}
We give here in detailed explaination for the case 
$(m,d) = (3,2).$
That is,
$\bF(x) \in \Sy^3\R[t]$
with entries are univariate polynomials of degree 2,
and
their coefficient vectors are randomly chosen in $(0,1)^{4}.$
\textsf{OCTAVE} produces  
the $(i,j)$-entries of $\bF$ 
in form $\bF[i,j]$  as follows
$$
\begin{array}{lllllllllll}
\bF & =
\\
  &\{ \\
 & [1,1] &=
     0.9239908 &  0.0092346 &  0.6041765
	\\
 & [2,1] &=
     0.38464   & 0.29627 &  0.56127
	\\
 & [3,1] &=
     0.69889   & 0.43163 &  0.20554
	\\
 & [1,2] &=
     0.38464   & 0.29627 &  0.56127
	\\
 & [2,2] &=
     0.42587   & 0.32317 &  0.57895
	\\
 & [3,2] &=
      0.91937  & 0.16887 &  0.48818
	\\
 & [1,3] &=
     0.69889   & 0.43163 &  0.20554
	\\
 & [2,3] &=
     0.91937  & 0.16887 &  0.48818
	\\
 & [3,3]& =
     0.77315   & 0.19066 &  0.24987
	\\
 &\}
\end{array}
$$
Algorithm 1 
gives us the polynomial $b$ of degree 16,
$\bX_p$'s and $\bX_m$'s entries are of degree 8, 
and diagonal elements' degrees are all 18.

We next give a particular example that is also returned later.

\begin{exam}\label{exam:1}
\rm
We  start with an example of $2\times 2$ univariate-polynomial matrix 
$$
\bF = 
\begin{bmatrix}
f & g \\ g & h
\end{bmatrix}, \quad
f(t) = t^2,\
g(t) = t(t+1),\
h(t) = (t+1)^2.
$$
Applying the algorithm we then obtain
$$
b = t^4, \enskip
\bX_+ =
\begin{bmatrix}
t^2 & 0\\ t(t+1)& t^2
\end{bmatrix}, \enskip
\bX_- =
\begin{bmatrix}
t^2 & 0\\ -t(t+1) & t^2
\end{bmatrix}, \enskip
\bD =
\begin{bmatrix}
t^6 & 0\\ 0 & 0
\end{bmatrix}.
$$

In addition,
it is not hard to see $\bF$ is positive semidefinite on $\R$ and 
$$
\{
t\in \R|\ \bF(t) \succeq 0, \forall t\in \R
\}
=
\{
t\in \R|\ \alpha_0= t^2\geq 0,\  \forall t\in \R
\} 
= \R,
$$
which illustrates \eqref{eq:kgSets}.
\end{exam}

\subsubsection{Multivariate polynomials}

To explain the tests in this situation, 
we need some more notations. 
Fix three positive numbers $m,n$ and $d.$
Let
\begin{align}
\Omega:=\Omega_{n,d}& = \{\alpha \in \Z_+^n:\ |\alpha| := \sum_{i=1}^n \alpha_i \leq d\},\\
\Gamma:=\Gamma_{n,d} &= \Omega(n,d) + \Omega(n,d) 
             = \{\gamma \in \Z_+^n:\ |\gamma| \leq 2d\}\\ \notag
            & = \Omega_{n,2d}.
\end{align}
The cardinalities of $\Omega$ and $\Gamma$ are well-known and are determined by
$|\Omega| = \binom{n+d}{n}$ 
and 
$|\Gamma| = \binom{n+2d}{n},$
respectively.

Our experiments in this situation
admits the 
\textit{lexicographic order} 
of monomials:
\begin{align*}
x^\alpha \leqq_{lex} x^\beta  \enskip
& \Longleftrightarrow \enskip
\alpha \leqq_{lex} \beta, \qquad (\alpha, \beta\in \Omega_{n,l})\\ \enskip
& \Longleftrightarrow \enskip
\alpha - \beta =(0, \ldots,0, \alpha_i -\beta_i,  \ldots, \alpha_l-\beta_l), \quad \alpha_i-\beta_i >0.
\end{align*}
For example, for
	$m=1,$ $n=2,$ $d=6$ then
	$$
	\Omega_{2,6} =\{(0,0), (0,1), \ldots, (0,6),
	   (1,0), \ldots, (1,5), \ldots, (5,0) , (5,1), (6,0)\}.
	$$ 
	The polynomial's coefficient vector is expressed with respect to the ``lex'' order as in $\Omega.$
	In particular, the Motzkin polynomial mentioned earlier will be  written as
	$
	f^M(x,y)= 1 -3x^2y^2 + x^2y^4 + x^4y^2,
	$
and its coefficient vector will be stored in \textsf{OCTAVE} as
$$
\mathbf{f}^{M} 
= [\overbrace{1}^{(0,0)}
   \enskip 0
   \enskip \ldots
   \enskip 0
    \enskip \overbrace{-3}^{(2,2)}
       \enskip 0
       \ldots
          \enskip 0
    \enskip \overbrace{1}^{(2,4)}
       \enskip 0
       \ldots
          \enskip 0
    \enskip \overbrace{1}^{(4,2)}
       \enskip 0
       \ldots
          \enskip 0                  
   ].
$$

It is worth mentioning that the lengths of multivariate polynomials' coefficient vectors rapidly increase as $n$ and $d$ are large.
This makes the size of the present problem rapidly large as well.

\begin{exam}\label{exam:2}
\rm 
Consider the matrix,
given in \cite{Schmudgen09},
$
\bF(x,y) 
= 
\begin{bmatrix}
1 + x^4y^2  & xy \\
xy          & 1 + x^2y^4
\end{bmatrix}.
$
Algorithm 1 produces $b= (1+x^4y^2)^2$ and
\begin{align*}
\bX_{\pm} &=
\begin{bmatrix}
1+x^4y^2  & 0\\
\pm xy    & 1+x^4y^2
\end{bmatrix}, \\
\bD &=
\begin{bmatrix}
(1+x^4y^2)^3  & 0 \\
0             & (1+x^4y^2)(1-x^2y^2+ x^2y^4 + x^4y^2 + x^6y^6 )
\end{bmatrix},
\end{align*}
where the coefficient vectors are ``full-length''.
For example, for
$b= (1+x^4y^2)^2 =  1 + 2x^4y^2 +x^8y^4$ above has degree 12,
and \textsf{OCTAVE} 
produces its coefficient vector 
of length $|\Omega_{2,12}|= 91$
with most zero entries except for the corresponding coefficients of $x^0y^0, x^4y^2, x^8y^4$ are $1,2, 1,$ respectively. 
\end{exam}

\section{Rank matrix minimization model (\ref{eq:Frmp}) }
\label{sec:RankMod}

\subsection{Algorithm}
As mentioned earlier,
the problem (\ref{eq:Frmp}) is solved by 
finding a matrix $X \in  \R^{m\times p}$ step by step for 
$\rk(X)= 1, 2, \ldots,\min\{m,p\}$ such that 
$\phi(X) = u.$  
The search of $X$ believes in
the gLM-method \cite{r665, cot}.

The Levenberg-Marquardt method \cite{r653} is a famous method that is used to solve the least squares problems.
For convenience to the readers, we summarize this method as follows.
A ``real data'' least square problem 
minimizes a real function $f(x),$ $x\in \R^n,$ 
where 
$$
f(x) = \| F(x)\|_2^2 ,  \quad 
F(x) = \left[ F_1(x)\enskip \ldots  \enskip F_l(x) \right]^\mathrm{T} \in \R^l, \quad \forall x\in \R^n,
$$
and $F_j$'s are continuously differentiable.
According to \cite{r653},  
starting with an initial point $x_0 \in \R^n,$
the method finds a sequence of points $\{x_k\} \subset \R^n$ that converges to a minimizer of $f.$
At each step $k,$ the next point $x_{k+1}$ is determined by applying 
the first-order
approximation of $f$ over an appropriately closed ``hyperellipsoid'' with center $x_k.$ 
More precisely,
one first approximates  $f(x_k+p)$ by its first order approximation
\begin{equation}\label{eq:LMstep}
F(x_k+p) \simeq F(x_k) + \jac F(x_k)p, \quad p\in \R^n.
\end{equation}
Then, one minimizes the approximating function 
$F(x_k) + \jac F(x_k)p$ over  a hyperellipsoid 
$E_k= \{p=(p_1, \ldots, p_n):\ \sum_{i=1}^n (d_i^{(k)})^2 p_i^2 \leq \Delta_k\}$
to search 
a descent direction  $p_*$ (a minimizer of the approximating function).
Either the step bound $\Delta_k$ or 
the scaling parameters $d_i^{(k)}$ will be updated after each step $k,$
depending upon the difference between $f(x_k+p_*)$ and $f(x_k).$
The next point $x_{k+1}$ will then be decided to be $x_k+p_*$ or exactly equal to $x_k.$
Roughly speaking, 
this method combines the advantages of the two well-known optimization methods: Gradient-descent and Gauss-Newton ones.
This setting has already been implemented in several technology environments, for example in \textsc{Matlab} by calling the 
``\texttt{lsqnonlin}'' or ``\texttt{fsolve}'' function.

In this situation,
we consider the least squares problem 
with respect to
the function $F: \R^{m\times p}  \rightarrow  \R^l,$
for appropriate integer number $p,$ whose coordinate functions are defined as
\begin{equation}\label{eq_genLSP}
F_i(X) = \phi_i(X) -u_i,  \quad \forall i =1,\ldots , l, 
\quad \forall X\in \R^{m\times p}.
\end{equation}
As mentioned in \eqref{eq:LMstep},
the Jacobian matrix of the input function $F$ is needed:
 $$
 \jac F 
 =
\begin{bmatrix}
\frac{\partial F_{1}(X)}{\partial x_{11}} & \frac{\partial F_{1}(X)}{\partial x_{21}} & \ldots & \frac{\partial F_{1}(X)}{\partial x_{mp}}\\
\frac{\partial F_{2}(X)}{\partial x_{11}} & \frac{\partial F_{2}(X)}{\partial x_{21}} & \ldots & \frac{\partial F_{2}(X)}{\partial x_{mp}}\\
				\vdots								 & 						\vdots						  & \ddots & \vdots\\
\frac{\partial F_{l}(X)}{\partial x_{11}} & \frac{\partial F_{l}(X)}{\partial x_{21}} & \ldots & \frac{\partial F_{l}(X)}{\partial x_{mp}}\\
\end{bmatrix}
\in \R^{l\times mp}, 
 $$ 
where 
$
\dfrac{\partial F_k}{\partial x_{ij}}
$
denotes the partial derivative of $F_k$ with respect to $x_{ij}.$
It is mentioned that 
$
\dfrac{\partial F_k}{\partial x_{ij}}
= \dfrac{\partial \phi_k}{\partial x_{ij}}.
$

Two tolerances (small enough) below are need 
at each Levenberg-Marquardt iteration:
\begin{enumerate} [$\bullet$]
\item 
	the step tolerance $\hat{\tau}:$   at the LM step $k$th, we rely on that a numerical solution $X_{[k]}$ to system (\ref{eq_genLSP}) exists if
	\begin{equation} \label{e_tol_iter} 
	\left| \frac{\|F(X_{[k]})\|_2^2-\|F(X_{[k-1]})\|_2^2 }{\| F(X_{[k-1]}) \|_2^2} \right|  \leq  \hat{\tau};
	\end{equation}
\item
	the residual tolerance $\tau:$ we rely on that the least squares problem has a numerical solution if 
	$\| F(X)\|_2< \tau.$
\end{enumerate}

\noindent
\rule[0ex]{\linewidth}{1.5pt}
\noindent \textbf{Algorithm 2}.  
Solving the problem \eqref{eq:Frmp}.
 \\ 
\rule[0ex]{\linewidth}{1.5pt}
\noindent
\begin{tabular}{ l  p{0.82\textwidth}}
{\textsc{Input}}:   & Vector $u=(u_1,\ldots, u_l) \in \R^l$ and function $\phi.$ 
 \\	
{\textsc{Output}}: & a solution $X\in \R^{m\times p}$  to (\ref{eq:Frmp}) .
 \\
& \\
\ \ \textit{Step 1}.			&	Set $r=1.$
\\
\ \ \textit{Step 2}.			&
	 Solve the system $\phi(X) = u,$ 
	 by applying the gLM-method to the function defined as in (\ref{eq_genLSP}). 
	 Let $X_{[k]}$ be a numerical solution determined by (\ref{e_tol_iter}) with respect to the tolerance $\hat{\tau}$ at some LM-iteration $k.$ Compute $F(X_{[k]}).$ 
\\
\\
\ \ Step 3.   &
	 If $\|F(X_{[k]})\|_2< \tau$ 
	 then $X = X_{[k]}$ is a numerical solution and stop.
	\newline
	Else, 
	\newline
			\mbox{\quad }  3.1. Set $r=r+1.$
	\newline  
         \mbox{\quad }  3.2. Go to Step 2.         
\end{tabular}
\vskip0mm
\noindent
\rule[0ex]{\linewidth}{1.5pt}

\subsection{Affine rank minimization problem over psd-scalar matrices}

\subsubsection{Reformulation}
 \label{sec:psdMatReform}

As mentioned earlier,
the coefficients of a sum of squares scalar-polynomial 
linearly depend on its Gram matrix's entries.
Finding a low rank Gram matrix of a sos scalar-polynomial  
hence 
leads to
an affine rank minimization problem over positive semidefinite matrices:
\begin{equation}\label{eq:Asrmp}
\min \left\{\rk(X)|\ X\in \Sy_+^m\R, \quad \ell(X) = b  \right\},
\end{equation}
where 
$\ell: \Sy^m\R \rightarrow \R^l$
is a linear map and 
$b\in \R^l$ is given.
It is clear that the above problem is not in form \eqref{eq:Frmp}.
It is well-known that any psd-matrix $X$ can be defined by its Cholesky factor $Y \in \Sy^m\R:$
$X= YY^\mathrm{T}.$
Therefore,
by expressing the linear map 
$\ell$
via $l$ matrices
$A_i \in \Sy^m\R:$
$$
\ell(X ) = [\tr(A_1 ^\mathrm{T} X) \quad  \ldots \quad \tr(A_l^\mathrm{T} X)]^\mathrm{T}, \quad 
\forall X\in \Sy^m\R, 
$$
we can define 
an objective function 
$F: \Sy^m\R \rightarrow \R$
that can 
apply LM-method \cite{r653, r665, cot} as:
$$
F_i(X) = \tr(A_i^\mathrm{T} YY^\mathrm{T})- b_i, \quad i=1, \ldots , l.
$$
Moreover,
if $X\in \Sy_+^m\R$ is needed of rank $r$ then so does
$Y\in \R^{m\times r}.$
The
problem (\ref{eq:Asrmp})
can be cast as
\begin{equation}\label{eq:Qsrmp} 
\begin{array}{llll}
\mbox{minimize} & \rk(Y) \\
\mbox{subject to}& \\
								&Y\in \R^{n\times r},\\
								& [\tr(A_1^\mathrm{T} YY^\mathrm{T})\quad\ldots \quad \tr(A_l^\mathrm{T} YY^\mathrm{T})]^\mathrm{T}=\ell(YY^\mathrm{T}) = b
\end{array}
\end{equation}
which has the form \eqref{eq:Frmp}.
In this situation,
the Jacobian  matrix of $F$ 
can be directly computed
as follows
$$
\jac (F) 
= \frac{\partial F}{\partial Y} 
= \begin{bmatrix}
	\frac{\partial F_1}{\partial Y} \\ 
	\vdots \\
	\frac{\partial F_l}{\partial Y}  
\end{bmatrix}
\in \R^{l\times nr}, 
$$
where
\begin{align*}
\frac{\partial F_i}{\partial Y} 
&=  \frac{\partial \tr(A_i^\mathrm{T} Y Y^\mathrm{T} )}{\partial Y}
= 2[\mathrm{vec}(A_iY)]^\mathrm{T}. 
\end{align*}
Here ``$\mathrm{vec}(X)$''
means the column vector that is 
built by stacking the X's columns as usual.

\subsubsection{Numerical tests} \label{sec:NumTestPosMat}

It is very well-known 
from the theoretical point of view that
the rank matrix minimization problem,
say \textit{RM-problem}, 
is NP-hard.
One reason 
might be the nonconvexity of the rank function,  
so that 
there has not been any method directly solve this one in literature.
A good way for solving the RM-problem  over positive semidefinite matrices  is to 
relax the rank function into
nuclear norm function (see, eg., in \cite{RFP10,MP97}).
The resulting problem
is then a semidefinite program \cite{p878}
and can be efficiently solved by SDP solvers. 
Very recently,
Huang and Wolkowicz \cite{Huang2017}
have proposed a method, 
that combines a facial reduction and low rank structure of semidefinite embedding,
to solve a matrix completion problem.
Another method for solving the NNM-problem over positive semidefinite matrices 
was proposed in \cite{Ma2011} by using 
modified
fixed point continuation method. 
Unfortunately,
resulting matrices given by SDP solver are usually full rank.

The experiments in this section are also implemented in 
\textsf{OCTAVE},
where some codes are 
translated 
from which of
COT \cite{cot}  in \textsc{Matlab}.
The input matrices $A_1, \ldots, A_k$ and $b$ 
have entries 
randomly chosen in (0,1).
Below we see that the resulting matrices 
given by gLM-method are low-rank.

\begin{table}[h]
\centerline{
\begin{tabular}{l 	l  |   c |c |  c| c c}
\\
 $n$    	&  $k$	& $\rk X$	 			&   $\#$ iter.	 err.  			&  backward err.          & time (seconds)
\\ \hline \hline
100 	& 50   & 2 	& 12		& 4.0e-15	& 0.3 \\
100 	& 100  & 2 	& 12	 	& 2.9e-15	& 0.54	\\
200 	& 200  & 2 	& 13		& 3.7e-15	& 2.91 \\
500 	& 200  & 2 	& 15		& 4.6e-15	& 37.1 \\
500 	& 400  & 2 	& 15		& 3.3e-14	& 67.7 \\
800 	& 500  & 2 	& 15		& 8.0e-14	& 309.9 \\
\end{tabular}
}
\caption{Numerical tests for the problem (\ref{eq:Qsrmp}) by applying the gLM-method implemented in \textsf{OCTAVE}.
} \label{table_srmp3} 
\end{table}

\section{Sosrf-based 
representations for some psd-polynomial matrices}
\label{sec:sosrf-based}

\subsection{Sosrf-based representations} \label{sec:sosrf} 
In this section 
we present and prove 
sosrf-based and/or sos-based decompositions of a psd-polynomial matrix defined on the sets: 
$\R$ and its special subset, $\R^n,$
and strips $[a,b] \times \R.$
The proofs are
based on the Schm\"{u}dgen's representation described in Proposition \ref{prop:Schmud} 
and 
Artin's Theorem 
\cite{Artin27} 
solving Hilbert's 17th problem 
stating the existence of
sosrf-representations of nonnegative scalar polynomials on $\R^n.$

The following result is a direct corollary of Proposition \ref{prop:Schmud},
its proof was presented in \cite{Schmudgen09} but  we would rewritten here 
that makes the readers more easily understand our proofs below.

\begin{prop}{\rm \cite[Proposition 11]{Schmudgen09}} \label{prop:coroSchmud}
For each symmetric (polynomial) matrix $\bF \in \Sy^m(\R[x]),$
the following statements are equivalent:
\begin{enumerate}[(i)]
\item[\rm (i)]
	$\mathbf{F} \succeq 0$ (on $\R^n$).
\item[\rm (ii)]
	$b^2\mathbf{F}$ are sum of Hermitian squares	 for some polynomial $b,$ that is 
	\begin{equation}\label{eq:sohs1}
		b^2 \mathbf{F} = \sum_{i=1}^r \mathbf{A}_i^\mathrm{T} \mathbf{A}_i,
	\end{equation}
	where $\mathbf{A}_i$ are polynomial matrices.
\end{enumerate}
\end{prop}
\begin{proof}
It is not hard to see $\bF \succeq 0$ on $\R^n$ if $p^2 \bF = \sum_{i=1}^r \mathbf{A}_i^\mathrm{T} \mathbf{A}_i.$

Conversely, we consider two cases of $\bF$ as follows:

(a) if $\bF$ is diagonal and  
$ \bF= \mathbf{D}:=\diag(d_1, \ldots, d_m) \succeq 0$ 
on $\R^n.$
Then,
denote $e_i$ the $i$th unit vector in $\R^n,$
we have
$$
d_i = e_i^\mathrm{T} \bD e_i  \mbox{ is nonnegative on } \R^n \mbox{ for every } i=1, \ldots, m.
$$ 
It thus follows from Artin's Theorem solving Hilbert's 17th problem
that 
$$
p_i^2 d_i = \sum_{j=1}^{r_j} g_{ij}^2,
$$
for some polynomials $p_i, g_{ij}\in \R[x].$
Set
$p = p_1\ldots p_r,$
$h_{ij} = g_{ij}\prod\limits_{j\neq i}{p_j}$
and
$A_j=  \diag(h_{1j}, \ldots, h_{rj})$
we have
$
p^2 \bD = \sum_{i=1}^k \bA_i^\mathrm{T} \bA_i.
$

(b) if $\bF$ is not diagonal then from (\ref{eq:SchmudDec}),
	$b^2 \bF = \bX_+ \bD \bX_+^\mathrm{T}$
	we are done by applying the  representation of $\bD$ above.
\end{proof}

The following result shows that 
the number of sosrf-terms in Proposition \ref{prop:coroSchmud}
can be restricted to 2 if
polynomial matrices defined on $\R.$
\begin{prop}\label{prop_sosrmR}  
Let $\bF\in \Sy^m \R[x],$ $x\in \R.$
Then $\bF\succeq 0$ on $\R$ if and only if there exists 
a polynomial $b \neq 0$ and two polynomial matrices
$\bU,\bV \in \R[x]^{m\times m}$
such that 
$$
b^2 \bF = \bU^\mathrm{T} \bU + \bV^\mathrm{T} \bV.
$$
\end{prop}
\begin{proof}
We first prove the proposition for $\bF$ is diagonal,
$$\bF= \bD:= \diag(d_1, \ldots, d_m)\succeq 0,$$ 
where $d_i$'s are non-negative polynomials.
One notes that 
each univariate polynomial which is non-negative on the real line can be expressed as a sum of at  most two squares.
So
$d_i = a_i^2 + b_i^2,\ \forall i=1, \ldots, m$
where 
$a_i, b_i\in \R[x].$
We then have
\begin{equation}\label{eq:sosmp_real}
\bD = \bA^\mathrm{T}\bA + \bB^\mathrm{T} \bB, \quad
\bA = \diag(a_1, \ldots, a_m),\
\bB = \diag(b_1, \ldots, b_m).
\end{equation} 

The opposite direction is not hard to obtain  by
directly computing the summation of two diagonal matrices 
$\bA^\mathrm{T} \bA+ \bB^\mathrm{T} \bB.$

Now, suppose $\bF\in \Sy^m \R[x]$ is positive semidefinite on $\R^n.$
It follows from Proposition \ref{prop:Schmud}, one has 
$$
\bX_+ \bX_- = \bX_- \bX_+ = b I_m, \quad
b^2 \bF = \bX_+ \bD  \bX_+^\mathrm{T}, \quad
\bD = \bX_- \bF \bX_-^\mathrm{T},
$$
for some 
$b\in \R[x],$
$\bX_{-}, \bX_+ \in \R[x]^{m\times m} $
and $\bD$ is diagonal.
Applying the above proof to the matrix $\bD$ one obtains  
$$
b^2 \bF = (\bA \bX_+^\mathrm{T})^\mathrm{T} (\bA \bX_+^\mathrm{T}) + (\bB \bX_+^\mathrm{T})^\mathrm{T}(\bB \bX_+^\mathrm{T}),
$$
where $\bA, \bB$ are defined as in (\ref{eq:sosmp_real}),
and we are done.

Conversely,
if $b^2 \bF= \bU^\mathrm{T} \bU+ \bV^\mathrm{T} \bV$ then  
for each $x\in \R$ and 
for all $y\in \R^m$ one has
\begin{align*}
b(x)^2[y^\mathrm{T} \bF(x) y]
&= y^\mathrm{T} \left[(b(x)^2 \bF(x) \right] y 
= y^\mathrm{T}\left[ \bU(x)^\mathrm{T} \bU(x)+ \bV(x)^\mathrm{T} \bV(x) \right] y 
\geq 0.
\end{align*}
This yields 
$y^\mathrm{T} \bF(x) y \geq 0$ for all $y\in \R^m$
and hence
$\bF(x) \succeq 0$ for all $x\in \R.$
\end{proof}

On the interval $[0, +\infty)$ we have the following.
%
%

\begin{prop}\label{prop_GenInf}
Let $\bF\in \Sy^m\R[x],$ $x\in \R.$
Then $\bF\succeq 0$ on $[0, +\infty)$ if and only if there exists 
a (scalar-coefficient) polynomial $b \neq 0$ and two polynomial matrices
$\bU, \bV \in \R[x]^{m\times m}$
such that 
$$
b(x)^2 \bF(x) = \bU(x)^\mathrm{T} \bU(x) + \bV(x)^\mathrm{T} \bV(x) x.
$$
\end{prop}
\begin{proof}
Analogously to the proof of the previous proposition, 
we prove the case of diagonal matrices.
Suppose $\bD= \diag(d_1, \ldots, d_m)\succeq 0,$
that is
$d_i(x) \geq 0$ for all $x\in [0, +\infty).$ 
It is provided by the work in \cite{Brickman62} that
$d_i = a_i^2 + b_i^2x,$
$a_i,b_i\in \R[x],$
for all $ \forall i=1, \ldots, m.$
So
$$
\bD = \bA^\mathrm{T} \bA + \bB^\mathrm{T} \bB x, 
\mbox{ where }
\bA = \diag(a_1, \ldots, a_m),\
\bB = \diag(b_1, \ldots, b_m).
$$

The ``\textit{only if}'' part is easy to see
by directly computing the multiplication of diagonal matrices.

We now consider an arbitrary
 $\bF\succeq 0$  on $[0,+\infty).$
Proposition \ref{prop:Schmud} leads us to the following decomposition
$$
\bX_+ \bX_- = \bX_- \bX_+ = bI_m, \quad
b^2 \bF = \bX_+ \bD \bX_+^\mathrm{T}, \quad
\bD = \bX_- \bF \bX_-^\mathrm{T},
$$
for some 
$b\in \R[x],$
$\bX_{-}, \bX_+ \in \R[x]^{m\times m} $
and $\bD$ is diagonal.
Using the representation of $\bD$ above we obtain that
$$
b^2 \bF = (\bA \bX_+^\mathrm{T})^\mathrm{T} (\bA \bX_+^\mathrm{T}) + (\bB \bX_+^\mathrm{T})^\mathrm{T}(\bB \bX_+^\mathrm{T})x
:= \bU^\mathrm{T} \bU + \bV^\mathrm{T} \bV x.
$$

For the converse direction,
take $x\geq 0$ arbitrary, then 
for every $y\in \R^m$ one has
\begin{align*}
b(x)^2 [y^\mathrm{T} \bF(x) y]
& = y^\mathrm{T} (b^2 \bF) y 
= y^\mathrm{T}(\bU^\mathrm{T} \bU+ \bV^\mathrm{T} \bV x)y \\
&= y^\mathrm{T} \bU^\mathrm{T} \bU y + (y^\mathrm{T} \bV^\mathrm{T} \bV y)x
\geq 0.
\end{align*}
This implies 
$y^\mathrm{T} \bF(x) y \geq 0$ for all $y\in \R^m.$
Thus
$\bF(x) \succeq 0$ on $[0,+\infty).$
\end{proof}

We now consider the polynomial matrices positive semidefinite on $[a,b].$
\begin{prop} \label{prop_DiagAB}
Let $\bD\in \R[x]^{m\times m}$ be a diagonal univariate-polynomial matrix.
Then $\bD\succeq 0$ on $[a,b]$ 
if and only if there are diagonal (polynomial) matrices $\bA_i,$ $i=1, \ldots, 4$ such that 
\begin{align*}
\bD(x) &= \bA_1(x)^\mathrm{T} \bA_1(x) (x-a) + \bA_2(x)^\mathrm{T} \bA_2(x) (b-x) \\
     &\quad + \bA_3(x)^\mathrm{T} \bA(x)  + \bA_4(x)^\mathrm{T} \bA_4(x) (b-x)(x-a).
\end{align*}
\end{prop}

Similarly to the two previous cases we are provided the following result.
Its proof utilises the same techniques as in the four propositions above.
\begin{prop}\label{prop_GenAB}
Let $\bF\in \Sy^m \R[x].$
Then $\bF\succeq 0$ on $[a,b]$ if and only if there exists 
a (scalar-coefficient) polynomial $b \neq 0$ and for polynomial matrices
$\bU_1, \bU_2, \bU_3, \bU_4 \in \R[x]^{m\times m}$
such that 
\begin{align*}
b^2(x) \bF(x) &= \bU_1(x)^\mathrm{T} \bU_1(x) (x-a) + \bU_2(x)^\mathrm{T} \bU_2(x) (b-x) \\
  				   &\quad + \bU_3(x)^\mathrm{T} \bU(x)  + \bU_4(x)^\mathrm{T} \bU_4(x) (b-x)(x-a).
\end{align*}
\end{prop}
\begin{proof}
It is sufficient to prove for diagonal matrix $\bF.$
It is well known that (see, eg., \cite{Brickman62, r678})
a scalar polynomial $p $ of degree $d$ is nonnegative on $[a,b]$ if and only if there are two polynomials 
$p_1, p_2$ such that
\small{
 $$
 p(x)=
 \begin{cases} p_1(x)^2 +(x-a)(b-x)p_2(x)^2, \deg p_2+1=\deg p_1=d_1, & \mbox{if } d=2d_1,\\
                                                                              &                        \\
                    (x-a)p_1(x)^2+(b-x)p_2(x)^2, \deg p_1=\deg p_2=d_1,  & \mbox{if }  d=2d_1+1.
  \end{cases}
 $$
}

\noindent
By decomposing 
$\bD= \bD_1 + \bD_2$ 
where 
$\bD_1$'s diagonal entries
 (resp., $\bD_2$'s) 
are those of $\bD$ with odd (resp., even) degree
as the same position in $\bD,$
and are zeros otherwise.
Then applying the above representation of polynomial nonnegative on $[a,b]$
to diagonal entries of $\bD_1$ and $\bD_2,$
one can find $\bA_1, \bA_2$ with respect to $\bD_1$ and $\bA_3, \bA_4$ with respect to $\bD_2.$

The converse part is derived analogously to the two previous cases.
\end{proof}

On the strips $[a, b] \times \R$ we have the following.
\begin{prop} \label{prop_DiagStr}
Let $\bD\in \R[x]^{m\times m}$ be a diagonal matrix defined on the strip $[0,1]\times \R.$
Then $\bD\succeq 0$ on the strip $[0,1] \times \R$ 
if and only if there exist diagonal (polynomial) matrices $\bA_t,\bB_t,$ $t=1, \ldots, k$ such that 
$$
\bD(x) 
= \sum_{t=1}^k  \bA_t(x)^\mathrm{T} \bA_t(x) + x_1(1-x_1) \sum_{t=1}^k \bB_t(x)^\mathrm{T} \bB_t(x) .
$$
\end{prop}
\begin{proof}
Suppose $\bD= \diag(d_1, \ldots, d_m)\succeq 0$ on the strip $[0,1] \times \R.$
One can find in \cite{Marshall10} that
$$
d_i(x_1, x_2) = \sigma_i(x_1, x_2)^2 +x_1(1-x_1)\ \tau_i(x_1, x_2)^2 ,\ \forall i=1, \ldots, m,
$$
where 
$\sigma_i, \tau_i\in \R[x_1, x_2]$
are sums of squares:
\begin{align*}
\sigma_i &= \sum_{t=1}^k g_{ik}^2, \quad g_{it} \in \R[x_1, x_2],\\
\tau_i &= \sum_{t=1}^k h_{ik}^2, \quad h_{it} \in \R[x_1, x_2]. 
\end{align*}
We then take 
$\bA_t = \diag(g_{1t}, \ldots , g_{kt})$ and  
$\bB_t = \diag(h_{1t}, \ldots , h_{kt})$
so that the desired representation of $\bD$ is obtained.

The converse direction is done in the similar way to the previous proofs.
\end{proof}

\begin{prop}\label{prop_GenStr}
Given $\bF\in \Sy^m(\R[x_1, x_2]).$
Then $\bF\succeq 0$ on $[0,1] \times \R$ 
if and only if there exists
a (scalar-coefficient) polynomial $b \neq 0$ and polynomial matrices
$\bU_t, \bV_t \in \R[x_1, x_2]^{m\times m},$
$t=1, \ldots, k$
such that 
$$
b(x)^2 \bF(x) 
= \sum_{t=1}^k  \bA_t(x)^\mathrm{T} \bA_t(x) + x_1(1-x_1) \sum_{t=1}^k \bB_t(x)^\mathrm{T} \bB_t(x) .
$$
\end{prop}
\begin{proof}
This proposition is a consequence of Propositions \ref{prop:Schmud} and \ref{prop_DiagStr}.
\end{proof}

\subsection{Numerical tests}

Given a polynomial matrix $\bF \in \Sy^m \R[x].$
The idea to find sos-based/sosrf-based representations of $\bF$ is
\begin{itemize}
\item
	firstly apply Algorithm 1 to $\bF$ to obtain a diagonal polynomial matrix $\bD = \diag(d_1, \ldots, d_m)$ satisfying \eqref{eq:SchmudDec};
	
\item	
	next, find sosrf-representations of $d_1, \ldots
	, d_m,$
	with assumption that the denominators in sosrf-representations of $d_i$'s are all exponents of sums of coordinates functions,
	by applying the algorithm in \cite{LeVaBa15}.
\end{itemize}

In the second stage of our computation,
we will express polynomial matrices as matrix polynomials.
The implementation 
needs 
 to take advantage of a monomial order for them.
In the rest of this paper,
we prefer 
the \textit{lexicographic order} of monomials earlier mentioned.
We suppose that a $n$-variable matrix polynomial  $\bP(x)$ of degree $d,$
is always expressed
under
lexicographic order:
$$
\bP(x) = \sum_{\alpha\in \Omega} P_\alpha x^\alpha, 
			\enskip P_\alpha \in \Sy^m\R,
			\quad \forall x\in \R^n.
$$ 
With
$\bP(x)$ written 
as above,
let furthermore
 $\mathcal{P}$ be the corresponding column vector of matrix coefficients $P_\alpha,$
and
$\Pi(x):=\Pi_{n,d}(x)$
be the column vector 
of matrices 
$x^\alpha I_m$ 
with respect to ``$\leqq_{lex}$''.
Then $\bP(x)$ can be written as
(see also in, eg., \cite{p039, d617})
$$
\bP(x) = \Pi(x)^\mathrm{T} \mP = \mP^\mathrm{T} \Pi(x).
$$
For example, 
\begin{enumerate}[i)]
\item
	$m=n=1,$ one has $\Omega(1,d) = \{0,1,\cdots, d\}$ and
	$$
	a(x) = p_0 + p_1 x + \ldots + p_d x^d 
	=
	[p_0 \quad p_1 \ldots \quad p_d] [1 \quad x \quad  \ldots \quad x^d]^\mathrm{T}
	= \mathbf{p}^\mathrm{T} \Pi_{1,d}(x)  
	$$

\item
	$m=n=d=2,$ $\Omega =\{(0,0), (0,1), (0,2), (1,0), (1,1), (2,0)\}$ and
	\begin{align*}
	\Pi_{2,2}(x_1, x_2) &=
	         [I_2  \quad x_2 I_2 \quad x_2^2 I_2 \quad x_1 I_2\quad x_1x_2 I_2 \quad x_1^2 I_2]^\mathrm{T}	,\\
	 \mP &= [P_{00} \quad P_{01} \quad P_{02} \quad P_{10} \quad P_{11} \quad P_{20}]^\mathrm{T}, 
	 		\quad  P_{ij} \in \Sy^2\R.         
	\end{align*}
One can check  that 
	$$
	\bP(x) = \mP^\mathrm{T} \Pi_{2,2}(x) 
	=
	P_{00} + P_{01}x_2 + P_{02} x_2^2 + P_{10}x_1+ P_{11} x_1 x_2 + P_{20} x_1^2	.
	$$	
\end{enumerate}

\begin{exam}\label{exam:3}
\rm 
Reconsider the matrix in Example \ref{exam:2}:
$
\bF(x,y) 
= 
\begin{bmatrix}
1 + x^4y^2  & xy \\
xy          & 1 + x^2y^4
\end{bmatrix}.
$
It can be verified that $\bF\succeq 0$ on $\R^2$ since
$\bF_{11} = 1 +x^4y^2 > 0,$
$\bF_{22} = 1 +x^2y^4 > 0$
and
\begin{align*}
\det \bF 
= (1+x^4y^2)(1+x^2y^4) -x^2y^2
&= 1 + x^4y^2+x^2y^4 +x^6y^6 - x^2y^2\\
(\mbox{AM-GM inequality})\qquad
& \geq 3\sqrt[3]{x^6y^6} +x^6y^6 -x^2y^2\\
&= 2x^2y^2 + x^6y^6 \\
& \geq 0, \quad \forall (x,y) \in \R^2.
\end{align*}
Relying on the algorithm proposed in \cite{LeVaBa15},
we wish to obtain
two diagonal polynomials of $\bD$ from Example \ref{exam:2} 
are sums of squares of rational functions.
It is clear that the 1st-one is a sum of squares.
The 2nd-one is 
$(1+x^4y^2)\det \bF.$
We see that both 
$\det \bF$ and 
$(1+x^4y^2)\det \bF$
consist of the term $-x^2y^2$
and they do not contain the terms
$$
x^4y^4,\
x^2,\
x^4,\
y^4,\
y^2.
$$
But since
$$
x^2y^2 
= 1. x^2y^2 
= x. xy^2 
= x^2. y^2
= x^2y. y
= xy. xy,
$$
both
$\det \bF$ and 
$(1+x^4y^2)\det \bF$
are not sums of squares of polynomial matrices.

Applying the algorithm proposed in \cite{LeVaBa15} we obtain an approximation\footnote{
In fact we can directly compute that
$$
(1+x^2+y^2 ) \det \bF(x,y)
=\frac{7}{8} + y^2 +x^2 +(\sqrt{2} x^2y^2 -1)^2 + x^2y^6 +x^6(y^2+y^6+y^8) + x^8y^6.
$$
}
of  $(1+x^2+y^2 ) \det \bF(x,y)$
as
a sum of 4 squares.
A decomposition of $\bF$ as in Proposition \ref{prop_sosrmR} is thus derived.
It is worth mentioning that the algorithm in \cite{LeVaBa15} is implemented in \textsc{Matlab},
we have converted  these \textsc{Matlab} codes into \textsf{OCTAVE} for our present situation.
\end{exam} 

\section{Conclusion}
We have clearly expressed the Schm\"{u}dgen's diagonalization for polynomial matrices and have then implemented in \textsf{OCTAVE}.
We have then proposed an algorithm for representing a global positive semidefinite polynomial matrix as a sum of Hermitian squares of polynomial matrices.
Our method combines the Schm\"{u}dgen's diagonalization 
and 
an algorithm that approximates the diagonal elements of the resulting diagonal matrix as sums of squares, 
due to Reznick's idea.
This have been applied to some psd polynomial matrices on several sets such as 
$\R$ and its intervals,
and strips $[a,b]\times \R.$
Corresponding numerical illustrations with \textsf{OCTAVE} have been presented in the paper.
This might lead to other numerical diagonalizations of other psd polynomial matrices in the future.

\subsubsection*{Acknowledgement}  
The first author would like to thank Vietnam Institute for Advanced Study in
Mathematics (VIASM) for warm hospitality
during his visit.


\end{document}